\documentclass[12pt, reqno]{amsart}

\usepackage{times}
\usepackage[T1]{fontenc}
\usepackage{mathrsfs}
\usepackage{latexsym}
\usepackage[dvips]{graphics}
\usepackage{epsfig}
\usepackage{amsmath,amsfonts,amsthm,amssymb,amscd}
\input amssym.def
\input amssym.tex
\usepackage{color}

\newcommand\be{\begin{equation}}
\newcommand\ee{\end{equation}}
\newcommand\bea{\begin{eqnarray}}
\newcommand\eea{\end{eqnarray}}
\newcommand\bi{\begin{itemize}}
\newcommand\ei{\end{itemize}}
\newcommand\ben{\begin{enumerate}}
\newcommand\een{\end{enumerate}}

\newtheorem{thm}{Theorem}[section]

\newtheorem{cor}[thm]{Corollary}
\newtheorem{lem}[thm]{Lemma}
\newtheorem{prop}[thm]{Proposition}

\newtheorem{defi}[thm]{Definition}

\newtheorem{cla}[thm]{Claim}

\numberwithin{equation}{section}

\begin{document}

\title{An Operator-Valued Kantorovich Metric on Complete Metric Spaces}

\maketitle

\begin{center}Trubee Davison\footnote{Department of Mathematics \\ \indent University of Colorado \\ \indent Campus Box 395 \\ \indent Boulder, CO 80309 \\ \indent trubee.davison@colorado.edu }\end{center}

\begin{abstract} 

The Kantorovich metric provides a way of measuring the distance between two Borel probability measures on a metric space. This metric has a broad range of applications from bioinformatics to image processing, and is commonly linked to the optimal transport problem in computer science \cite{Deng} \cite{Villani}. Noteworthy to this paper will be the role of the Kantorovich metric in the study of iterated function systems, which are families of contractive mappings on a complete metric space. When the underlying metric space is compact, it is well known that the space of Borel probability measures on this metric space, equipped with the Kantorovich metric, constitutes a compact, and thus complete metric space. In previous work, we generalized the Kantorovich metric to operator-valued measures for a compact underlying metric space, and applied this generalized metric to the setting of iterated function systems \cite{Davison} \cite{Davison3} \cite{Davison4}. We note that the work of P. Jorgensen, K. Shuman, and K. Kornelson provided the framework for our application to this setting \cite{Jorgensen2} \cite{Kornelson} \cite{Jorgensen}. The situation when the underlying metric space is complete, but not necessarily compact, has been studied by A. Kravchenko in \cite{Kravchenko}. In this paper, we extend the results of Kravchenko to the generalized Kantorovich metric on operator-valued measures.

\end{abstract}

\tableofcontents

\noindent \keywords[Keywords: Kantorovich metric, Operator-valued measure, Hutchinson measure, Iterated Function System, Fixed point] \\
\\
\noindent [Mathematics subject classification 2010: 46C99 - 46L05] \\

\noindent [Publication note: The final publication is available at Springer via \\
https://doi.org/10.1007/s10440-018-0213-y.]

\section{Introduction} 

This is the third in a series of related papers (see \cite{Davison} and \cite{Davison4}) which discuss a generalized Kantorovich metric for operator-valued measures, and its application to the study of iterated function systems. Let $(Y,d)$ be a complete and separable metric space, and let $\mu$ and $\nu$ be two Borel probability measures on $Y$. We define the Kantorovich metric, $H$, between the two measures by

\begin{equation} \label{KantClass} H(\mu, \nu) = \sup_{f\in \text{Lip}_1(Y)} \left\{ \left| \int_Y f d\mu - \int_Y f d\nu \right| \right\}, \end{equation} \\

\noindent where $\text{Lip}_1(Y) = \{ f:Y \rightarrow \mathbb{R} : |f(x) - f(y)| \leq d(x,y) \text{ for all } x,y \in Y\}.$ 

An appealing feature of the Kantorovich metric is that it is a natural generalization of the underlying metric on $Y$. Indeed, the map $y \in Y \mapsto \delta_y$ is an injective metric space isometry, where $\delta_y$ is the delta measure measure at $y$. Furthermore, if $Y$ is compact, and if $Q(Y)$ is the collection of Borel probability measures on $Y$, we have the following important facts. 

\begin{enumerate} 

\item \label{classical} $(Q(Y), H)$ is compact. 

\item The topology induced by the metric $H$ on $Q(Y)$ coincides with the weak topology on $Q(Y)$.  

\end{enumerate} 

In the case that $(Y,d)$ is not compact, and thus possibly unbounded, the Kantorovich metric may not be finite on $Q(Y)$. To remedy this issue, we restrict the $H$ metric to $M(Y)$, which is defined to be the sub-collection of Borel probability measures $\nu$ on $Y$ such that $\int |f| d\nu < \infty$ for all $f \in \text{Lip}(Y)$, where $\text{Lip}(Y)$ is the collection of real-valued Lipschitz functions on $Y$. 

We say that a sequence of measures $\{\mu_n\}_{n=1}^{\infty} \subseteq M(Y)$ converges weakly to a measure $\mu \in M(Y)$ if 
$$\lim_{n \rightarrow \infty} \int_Y f d\mu_n = \int_Y f d\mu$$ 
\noindent for all $f \in C_b(Y)$, where $C_b(Y)$ denotes the collection of real-valued, bounded, and continuous functions on $Y$. The following result due to A. Kravchenko will be important to this paper. \\

\begin{thm} \cite{Kravchenko} \label{Kravthm} [Kravchenko]

\begin{enumerate} 

\item The metric space $(M(Y), H)$ is complete. 

\item The topology induced by the $H$ metric coincides with the weak topology on $M(Y)$ if and only if the metric space $Y$ is bounded.

\end{enumerate} 

\end{thm}

\noindent The main tool in proving the first part of the above theorem is the following proposition. Note that $\text{Lip}_b(Y)$ refers to the collection of real-valued, bounded Lipschitz functions on $Y$. 

\begin{prop} \cite{Kravchenko} \label{keylemma}[Kravchenko] Let $\{\mu_n\}_{n=1}^{\infty}$ be a sequence of Borel measures on the complete and separable metric space $Y$ such that $\mu_n(Y) = K < \infty$ for all $n =1,2,...$, and such that for all $f \in \text{Lip}_b(Y)$, the sequence of real numbers 
$$ \left\{ \int_Y f d\mu_n \right\}_{n=1}^{\infty}$$

\noindent is Cauchy.  Then there exists a Borel measure $\mu$ on $Y$ such that $\mu(Y) = K$, and such that the sequence $\{\mu_n\}_{n=1}^{\infty}$ converges weakly to $\mu$. 
\end{prop}

In previous work, we generalized the Kantorovich metric to operator-valued measures when the underlying metric space $(Y,d)$ is compact \cite{Davison} \cite{Davison3} \cite{Davison4}. Indeed, let $\mathcal{H}$ be a Hilbert space, and define $\mathcal{B}(\mathcal{H})$ to be the bounded operators on $\mathcal{H}$. Put $\mathscr{B}(Y)$ to be the collection of Borel measurable subsets of $Y$. 

\begin{defi} A positive operator $L \in \mathcal{B}(\mathcal{H})$ satisfies $\langle Lh, h \rangle \geq 0$ for all $h \in \mathcal{H}$.
\end{defi} 

\begin{defi} A positive operator-valued measure with respect to the pair $(Y,\mathcal{H})$ is a map $A: \mathscr{B}(Y) \rightarrow \mathcal{B}(\mathcal{H})$ such that:

\begin{itemize} 

\item $A(\Delta)$ is a positive operator in $\mathcal{B}(\mathcal{H})$ for all $\Delta \in \mathscr{B}(Y)$;

\item $A(\emptyset) = 0$ and $A(Y) = \text{id}_{\mathcal{H}}$ (the identity operator on $\mathcal{H}$);

\item If $\{ \Delta_n \}_{n=1}^{\infty}$ is a sequence of pairwise disjoint sets in $\mathscr{B}(Y)$, and if $g,h \in \mathcal{H}$, then
$$ \left \langle A\left( \bigcup_{n=1}^{\infty} \Delta_n \right)g , h \right \rangle = \sum_{n=1}^{\infty} \langle A(\Delta_n)g, h \rangle.$$

\end{itemize}

\end{defi}

\begin{defi} A projection $E \in \mathcal{B}(\mathcal{H})$ satisfies $E^2 = E$ (idempotent) and $E^* = E$ (self-adjoint).
\end{defi} 

\begin{defi} A projection-valued measure with respect to the pair $(Y,\mathcal{H})$ is a map $E: \mathscr{B}(Y) \rightarrow \mathcal{B}(\mathcal{H})$ such that 

\begin{itemize} 

\item $E(\Delta)$ is a projection in $\mathcal{B}(\mathcal{H})$ for all $\Delta \in \mathscr{B}(Y)$;

\item $E(\emptyset) = 0$ and $E(Y) = \text{id}_{\mathcal{H}}$ (the identity operator on $\mathcal{H}$);

\item $E(\Delta_1 \cap \Delta_2) = E(\Delta_1)E(\Delta_2)$ for all $\Delta_1, \Delta_2 \in \mathscr{B}(Y)$;

\item If $\{ \Delta_n \}_{n=1}^{\infty}$ is a sequence of pairwise disjoint sets in $\mathscr{B}(Y)$, and if $g,h \in \mathcal{H}$, then
$$ \left \langle E\left( \bigcup_{n=1}^{\infty} \Delta_n \right)g , h \right \rangle = \sum_{n=1}^{\infty} \langle E(\Delta_n)g, h \rangle.$$

\end{itemize} 
\end{defi}

Let $S(Y)$ be the collection of positive operator-valued measures with respect to the pair $(Y, \mathcal{H})$, and let $P(Y)$ be the collection of projection-valued measures with respect to the pair $(Y, \mathcal{H})$. Since projections are positive operators, we have the inclusion $P(Y) \subseteq S(Y)$. Define a metric $\rho$ on $S(Y)$ by 

\begin{equation} \label{genKant} \rho(A,B) = \sup_{f \in \text{Lip}_1(Y)}\left\{ \left|\left| \int_Y f dA - \int_Y f dB \right|\right| \right\}, \end{equation}

\noindent where $||\cdot||$ denotes the operator norm in $\mathcal{B}(\mathcal{H})$, and $A$ and $B$ are arbitrary members of $S(Y)$.  This is called the generalized Kantorovich metric, and we first defined this metric in \cite{Davison} and \cite{Davison3}. 

\begin{thm} \cite{Davison} \cite{Davison3} [Davison] \label{mainthm}
If $(Y,d)$ is compact, the metric space $(S(Y), \rho)$ is complete, and $P(Y)$ is a closed subset of $S(Y)$ with respect to the $\rho$ metric. As such, $(P(Y), \rho)$ is complete. 

\end{thm} 

We now provide some context for this generalized Kantorovich metric by relating it to the operator norm in $\mathcal{B}(\mathcal{H})$, and also by discussing a noteworthy application of this metric to quantum theory. 
Let $M > 0$ be some fixed constant, and let $B_M$ be the collection of all normal operators on $\mathcal{H}$ such that $||N|| \leq M$, where $||\cdot||$ denotes the operator norm of $N$.  We note that if $N \in B_M$, then the spectrum of $N$, denoted $\sigma(N)$, is contained in $B_0(M)$, the closed ball of radius $M$ in $\mathbb{C}$ centered at the origin.  By the Spectral Theorem for Normal Operators, there exists a one-to-one correspondence between $B_M$ and the collection of projection-valued measures with respect to the pair $(B_0(M), \mathcal{H})$.  That is, the map $N \in B_M \mapsto E \in P(B_0(M))$ is bijective, where $E$ satisfies 
$$N = \int_{\sigma(N)} z dE(z).$$   

Since $B_0(M)$ is compact, we conclude by Theorem \ref{mainthm} that $(P(B_0(M)), \rho)$ is a complete metric space.  Moreover, the $\rho$ metric induces a metric on $B_M$ in the following way.  If $N, A \in B_M$, define 
$$\Psi(N,A) = \rho(E,F),$$

\noindent where $N = \int_{\sigma(N)} z dE(z)$, and $A = \int_{\sigma(A)} z dF(z)$.  

The bijective correspondence $N \in B_M \mapsto E \in P(B_0(M))$ guarantees that $\Psi$ is a metric, and that the metric space $(B_M, \Psi)$ is complete.  By definition, a sequence $\{N_k\}_{k=1}^{\infty} \subseteq B_M$ converges to $N \in B_M$ in the $\Psi$ metric if and only if the corresponding sequence $\{E_k\}_{k=1}^{\infty} \subseteq B_0(M)$ converges to $E \in B_0(M)$ in the $\rho$ metric. The relationship between the $\Psi$ metric and operator norm is described below. 

\begin{prop} \label{cool} \cite{Davison} \cite{Davison3} [Davison] The topology induced by the $\Psi$ metric on $B_M$ coincides with the topology induced by the operator norm on $B_M$. 
\end{prop} 

Recall that if the metric space $Y$ is compact, then $(Q(Y), H)$ is compact. However, the metric space $(P(Y), \rho)$ is not compact. This hinges on the above proposition, and is described in our previous work \cite{Davison}. 

We now highlight an interesting application of the generalized Kantorovich metric to mathematical physics. Namely, this metric has been previously defined by R.F. Werner in the setting of quantum theory \cite{Werner}. A positive operator-valued measure is also called an observable in physics. Werner introduces the generalized Kantorovich metric as a tool for studying the position and momentum observables, which are central objects of study in quantum theory.

We initially defined the generalized Kantorovich metric to be used in the study of iterated function systems \cite{Davison} \cite{Davison3} \cite{Davison4}. Let $L: Y \rightarrow Y$ be a Lipschitz contraction on the complete and separable metric space $(Y,d)$. The map $L$ admits a unique fixed point $y \in Y$, meaning that $L(y) = y$.  This result is known as the Contraction Mapping Principle, or the Banach Fixed Point Theorem.  In 1981, J. Hutchinson published a seminal paper (see \cite{Hutchinson}), where he generalized the Contraction Mapping Principle to a finite family, $\mathcal{S} = \{\sigma_0,...,\sigma_{N-1}\}$, of Lipschitz contractions on $Y$, where $N \in \mathbb{N}$ is such that $N \geq 2$. Indeed, one can associate to $\mathcal{S}$ a unique compact subset $X \subseteq Y$ which is invariant under the $\mathcal{S}$, meaning that 

\begin{equation} \label{fractal} X = \bigcup_{i=0}^{N-1} \sigma_i(X). \end{equation} 

The finite family of Lipschitz contractions, $\mathcal{S} = \{\sigma_0,...,\sigma_{N-1}\}$, is called an iterated function system (IFS) on $Y$, and the compact invariant subset $X$ described above is called the self-similar fractal set, or attractor set, associated to the IFS. The attractor set can be realized as the support of a Borel probability measure, which we call the Hutchinson measure. This is described in the following theorem. 

\begin{thm} \cite{Hutchinson} \cite{Kravchenko}  [Hutchinson, Kravchenko] \label{ambient} The map $T: M(Y) \rightarrow M(Y)$ given by 

$$T(\nu) = \sum_{i=0}^{N-1} \frac{1}{N}\nu(\sigma_i^{-1}(\cdot)),$$

\noindent is a Lipschitz contraction on the complete metric space $(M(Y), H)$. By the Contraction Mapping Principle, there exists a unique Borel probability measure $\mu \in M(Y)$ such that $T(\mu) = \mu$. The support of $\mu$ is the attractor set $X$. 
\end{thm} 

Let $\mathcal{H} = L^2(X, \mu)$, where $X$ is the attractor set associated to $\mathcal{S}$, and $\mu$ is the corresponding Hutchinson measure. Following the work of P. Jorgensen, we define operators on $\mathcal{H}$ as follows:

\begin{center} $F_i: \mathcal{H} \rightarrow \mathcal{H}$ given by $\displaystyle{\phi \mapsto \frac{1}{\sqrt{N}} (\phi \circ \sigma_i)}$ \end{center} 

\noindent for all $i = 0, ..., N-1$. 

\begin{thm} \cite{Kornelson} \label{Kornelson} [Jorgensen, Kornelson, Shuman] The family of operators $\{F_i\}_{i=0}^{N-1}$ satisfy the operator identity 

$$\sum_{i=0}^{N-1} F_i^*F_i = \text{id}_{\mathcal{H}}. $$

\end{thm}

Consider the metric space, $(S(X), \rho)$, consisting of positive operator-valued measures with respect to the pair $(X, \mathcal{H})$. 

\begin{thm} \label{contraction} \cite{Davison4} [Davison] The map $V: S(X) \rightarrow S(X)$ given by $$B(\cdot) \mapsto \sum_{i=0}^{N-1} F_i^* B(\sigma_i^{-1}(\cdot))F_i$$ 

\noindent is a Lipschitz contraction in the $\rho$ metric. 

\end{thm}

\begin{cor} \label{fixedPOVM} \cite{Davison4} [Davison] There exists a unique positive operator-valued measure $A \in S(X)$ such that
\begin{equation} A(\cdot) = \sum_{i=0}^{N-1} F_i^* A(\sigma_i^{-1}(\cdot))F_i. \end{equation}
\end{cor} 

The proof of the above corollary hinges on Theorem \ref{contraction}, and the fact that $(S(X), \rho)$ is a complete metric space (which is Theorem \ref{mainthm}). Notably, the operator-valued measure $A$ generalizes the the Hutchinson measure in the sense that if $f = 1 \in \mathcal{H}$,  the scalar measure $A_{f,f}(\cdot) =\langle A(\cdot)f, f \rangle$ satisfies 
$$A_{f,f}(\cdot) = \mu(\cdot).$$ 

It is important to note that Jorgensen was the first to associate an operator-valued measure to an iterated function system, and thus he laid the groundwork for our study of this topic. Indeed, we have the following definition and result. 

\begin{defi} An IFS has non-essential overlap if 

$$\mu(\sigma_i(X) \cap \sigma_j(X)) = \emptyset$$
when $i \neq j$. 

\end{defi} 

\begin{thm}\cite{Jorgensen2} \cite{Jorgensen}  \label{uniquepvm} [Jorgensen] Suppose that an IFS has non-essential overlap, and that each member of the IFS is an injection. There exists a unique projection-valued measure $E$ with respect to the pair $(X, \mathcal{H})$ such that

$$E(\cdot) = \sum_{i=0}^{N-1} F_i^* E(\sigma_i^{-1}(\cdot)) F_i.$$

\end{thm}

Previously in \cite{Davison}, we presented an alternative proof to Theorem \ref{uniquepvm} using the Contraction Mapping Principle. This approach was further generalized to the case of an arbitrary IFS, which is described in Theorem \ref{contraction}, and Corollary \ref{fixedPOVM}. 

In this paper, we will consider the generalized Kantorovich metric when the underlying metric space is complete and separable (not necessarily compact). Namely, we will generalize Part (1) of Theorem \ref{Kravthm} to projection and positive operator-valued measures. An application of this result will be to restate Theorem \ref{contraction} in the following way. Notice that the map $T$ described in Theorem \ref{ambient} is a map on $M(Y)$, where $Y$ is the ambient metric space which contains the attractor set $X$ as a subset. The map $V$ described in Theorem \ref{contraction} is a map on $S(X)$, where $X$ is the attractor set (not the ambient space). We will extend $V$ to a map on $S_0(Y)$, a sub-collection of positive operator-valued measures with respect to the pair $(Y, L^2(Y,\mu))$, thereby providing a direct generalization of Theorem \ref{ambient} to operator-valued measures. For this to be possible, we note that the maps $F_i: L^2(X,\mu) \rightarrow L^2(X,\mu)$ defined above can extend to maps $F_i: L^2(Y, \mu) \rightarrow L^2(Y\mu)$ given by the same formula, and Theorem \ref{Kornelson} still holds true if we view these operators on $L^2(Y, \mu)$. 

\section{Results}

\subsection{A Sub-Collection of Operator-Valued Measures}  

Let $(Y, d)$ be a complete and separable metric space, and let $\mathcal{H}$ be a Hilbert space. Let $P_0(Y)$ be the collection of projection-valued measures with respect to the pair $(Y, \mathcal{H})$ with the following additional property:  If $E \in P_0(Y)$, then for all $f \in \text{Lip}(Y)$, there exists an $0 \leq M_{f,E} < \infty$ such that
$$\left| \int_Y f dE_{g,h} \right| \leq M_{f,E}||g||||h||,$$

\noindent for all $g,h \in \mathcal{H}$, and where $\text{Lip}(Y)$ denotes the collection of all real-valued Lipschitz functions on $Y$. An example of an element in $P_0(Y)$ would be a projection-valued measure $E$ such that $E(K) = {\bf{1}}_{\mathcal{H}}$, for $K$ a compact subset of $Y$.  In this case, note that $E(Y\setminus K) = 0$.  If $f \in \text{Lip}(Y)$, let $M_{f,E} = \max_{x \in K} |f(x)|$ and observe that
\begin{eqnarray*}
\left| \int_Y f dE_{g,h} \right| & \leq & \left| \int_K f dE_{g,h} \right| + \left| \int_{Y\setminus K} f dE_{g,h} \right| \\
& \leq &  \int_K |f| d|E_{g,h}| \\
& \leq & M_{f,E} ||g||||h||, 
\end{eqnarray*} 

\noindent for all $g, h \in \mathcal{H}$, which implies that $E \in P_0(Y)$.  We can define $S_0(Y)$ to be the positive operator-valued measures with respect to the pair $(Y, \mathcal{H})$ with this same additional property. We have the containment $P_0(Y) \subseteq S_0(Y)$, since projections are positive operators. 

On this sub-collection of positive operator-valued measures we will consider the generalized Kantorovich metric.  That is, for $A,B \in S_0(Y)$ define (exactly as before)
$$\rho(A,B) = \sup_{f \in \text{Lip}_1(Y)} \left \{ \left| \left| \int_Y f dA - \int_Y f dB \right| \right| \right \}.$$

\begin{thm} \cite{Conway} \label{boundedsesqui}[Theorem II.2.2 in Conway] Let $u: \mathcal{H} \times \mathcal{H} \rightarrow \mathbb{C}$ be a bounded sesquilinear form with bound $M$.  There exists a unique operator $A \in \mathcal{B}(\mathcal{H})$ such that $u(g,h) = \langle Ag, h \rangle$ for all $g,h \in \mathcal{H}$, and such that $||A|| \leq M$.
\end{thm}

We will now show that this metric is finite on $S_0(Y)$. To do this, we need to make a preliminary observation.  In particular, if $A \in S_0(Y)$ and $f \in \text{Lip}(Y)$, there exists by definition $0 \leq M_{f,A} < \infty$ such that
$$\left| \int_Y f dA_{g,h} \right| \leq M_{f,A}||g||||h||,$$

\noindent for all $g,h \in \mathcal{H}$.  This means that the map $[g,h] \mapsto \int_Y f dA_{g,h}$ is a bounded sesquilinear form.  By Theorem \ref{boundedsesqui}, there exists a bounded operator $\int f dA \in \mathcal{B}(\mathcal{H})$ such that
$$\left \langle \left( \int_Y f dA \right)g, h \right \rangle = \int _Y f dA_{g,h},$$

\noindent for all $g,h \in \mathcal{H}$, where $\left| \left| \int_Y f dA \right| \right| \leq M_{f,A}$.  

Let $A, B \in S_0(Y)$, $f \in \text{Lip}_1(Y)$, and $x_0 \in Y$.  Then
\begin{eqnarray*} \label{finite}
\left| \left| \int_Y f dA - \int_Y f dB \right| \right| & = & \left| \left| \int_Y f dA - f(x_0)\text{id}_{\mathcal{H}} + f(x_0)\text{id}_{\mathcal{H}} - \int_Y f dB \right| \right| \nonumber \\ 
								      & = & \left| \left| \int_Y f dA - \int_Y f(x_0) dA - \left( \int_Y f dB - \int_Y f(x_0) dB \right) \right| \right| \nonumber \\ 
								      & \leq & \left| \left| \int_Y (f - f(x_0)) dA \right| \right| + \left| \left| \int_Y \left(f - f(x_0)\right) dB \right| \right| \nonumber 
\end{eqnarray*}

\noindent Let $h \in \mathcal{H}$ with $||h||=1$.  Then

\begin{eqnarray*} 
\left| \left \langle \left(\int_Y (f(x) - f(x_0)) dA \right)h, h \right \rangle \right| & = & \left| \int_Y (f(x) - f(x_0)) dA_{h,h}(x) \right| \\ 
& \leq & \int_Y |f(x) - f(x_0)| dA_{h,h}(x) \\
& \leq & \int_Y d(x,x_0) dA_{h,h}(x) \\
& \leq & M_{d(x,x_0),A} ||h||^2 \\
& = & M_{d(x,x_0),A},
\end{eqnarray*}

\noindent where $M_{d(x,x_0),A} \geq 0$ is the non-negative number associated to the $\text{Lip}(Y)$ function $d(x,x_0)$ and the positive operator-valued measure $A$.  Hence, 
$$\left| \left| \int_Y (f - f(x_0)) dA \right| \right| \leq M_{d(x,x_0), A}.$$

\noindent Similarly, there exists an $M_{d(x,x_0),B} \geq 0$ such that
$$\left| \left| \int_Y (f - f(x_0)) dB \right| \right| \leq M_{d(x,x_0),B}.$$

\noindent Since $M_{d(x,x_0),A}$ and $M_{d(x,x_0),B}$ do not depend on the choice of $f \in \text{Lip}_1(Y)$, $\rho(A,B) \leq M_{d(x,x_0),A} + M_{d(x,x_0),B} < \infty.$

\subsection{The Metric Space $(P_0(Y), \rho)$ is Complete} 

In this section, we will show that the metric space $(P_0(Y), \rho)$ is complete.  We will rely on Proposition \ref{keylemma}.  We will also use the following lemma, which can be found in the proof of Proposition 1 in \cite{Berberian}. 

\begin{lem} \cite{Berberian} \label{pos}[Proposition 1 in Berberian] Let $\{B_n\}_{n=1}^{\infty}$ be a sequence of positive operators on the Hilbert space $\mathcal{H}$ such that $||B_n|| \leq M$ for all $n = 1, 2,...$, and such that for all $h \in \mathcal{H}$, $\lim_{n \rightarrow \infty} \langle B_n h, h \rangle = 0$.  Then $\lim_{n \rightarrow \infty} ||B_n h|| =0$.
\end{lem}

\begin{proof} Let $h \in \mathcal{H}$.  Note that $||B_n h||^2 = \langle B_n h, B_n h \rangle = |\langle B_n h, g \rangle|$ where $g = B_n h$.  By a generalized Schwarz inequality, 
$$0 \leq |\langle B_n h, g \rangle| \leq \langle B_n h, h \rangle^{\frac{1}{2}} \langle B_n g, g \rangle^{\frac{1}{2}} \leq \langle B_n h, h \rangle^{\frac{1}{2}} ( ||B_n g || ||g||)^{\frac{1}{2}} \leq $$ 
$$\langle B_n h, h \rangle^{\frac{1}{2}} ( ||B_n||^3 ||h||^2)^{\frac{1}{2}} \leq \langle B_n h, h \rangle^{\frac{1}{2}}M^{\frac{3}{2}} ||h||.$$

\noindent By assumtion, $\lim_{n \rightarrow \infty} \langle B_n h, h \rangle = 0$.  Hence, $\lim_{n \rightarrow \infty} \langle B_n h, h \rangle^{\frac{1}{2}}M^{\frac{3}{2}} ||h|| = 0$, which implies that 
$\lim_{n \rightarrow \infty} ||B_n h|| =0$. 
\end{proof} 

\begin{thm} \label{P_0complete} \cite{Davison3} [Davison] The metric space $(P_0(Y), \rho)$ is complete. 
\end{thm} 

\begin{proof} We note that the proof of this result uses some of the same techniques as in the proof of the analogous result by Kravchenko \cite{Kravchenko} that $(M(Y), H)$ is complete (see Theorem \ref{Kravthm}).  Suppose that $\{E_n\}_{n=1}^{\infty}$ is a Cauchy sequence of elements in $(P_0(Y), \rho)$.  We want to find an $E \in (P_0(Y), \rho)$ such that $E_n \rightarrow E$ in the $\rho$ metric.

\begin{cla} \label{Cauchy} Let $h \in \mathcal{H}$ and $f \in \text{Lip}(Y)$.  The sequence $\{\int_Y f dE_{n_{h,h}}\}_{n=1}^{\infty}$ is a Cauchy sequence of real numbers.  
\end{cla}

\noindent Proof of Claim: This claim follows from the observation that there exists a $T > 0$ such that $\frac{f}{T} \in \text{Lip}_1(Y)$.  Hence
$$0 \leq \left| \int_Y \frac{f}{T} dE_{n_{h,h}} - \int_Y \frac{f}{T} dE_{m_{h,h}} \right| = \left| \left \langle \left( \int_Y \frac{f}{T} dE_n - \int_Y \frac{f}{T} dE_m \right) h, h \right \rangle \right| \leq $$
$$\left| \left| \int_Y \frac{f}{T} dE_n - \int_Y \frac{f}{T} dE_m \right| \right| ||h||^2 \leq \rho(E_n,E_m) ||h||^2 \rightarrow 0$$

\noindent as $m,n \rightarrow \infty$.  Since $||h||$ and $T$ are fixed, 
$$\left| \int_Y f dE_{n_{h,h}} - \int_Y f dE_{m_{h,h}} \right| \rightarrow 0$$

\noindent as $m,n \rightarrow \infty$.  This proves the claim.  

Observe that $E_{n_{h,h}}(Y) = \langle E_n(Y)h,h \rangle = ||h||^2$ for all $n = 1,2,...$.  Since $\text{Lip}_b(Y) \subseteq \text{Lip}(Y)$, we can use Proposition \ref{keylemma} to conclude that there exists a Borel measure $\mu_{h,h}$ on $Y$ such that $\mu_{h,h}(Y) = ||h||^2$, and such that $E_{n_{h,h}}$ converges to $\mu_{h,h}$ in the weak topology.  That is, for all $f \in C_b(Y)$ we have that
$$ \lim_{n \rightarrow \infty} \int_Y f dE_{n_{h,h}} = \int_Y f d\mu_{h,h}.$$

\noindent For $g, h \in \mathcal{H}$, we want to define $\mu_{g,h}$ such that 
\begin{equation} \label{weakconv} \lim_{n \rightarrow \infty} \int_Y f dE_{n_{g,h}} = \int_Y f d\mu_{g,h} \end{equation}
for all $f \in C_b(Y)$.  To this end, let $g, h \in \mathcal{H}$.  If $E_{n_{g,h}} = \text{Re}E_{n_{g,h}} + i\text{Im}E_{n_{g,h}},$ we can calculate that 
$$\mathrm{Re}E_{n_{g,h}} = \frac{1}{2}(E_{n_{g+h,g+h}} - E_{n_{g,g}} -E_{n_{h,h}})$$ 

\noindent and
$$\text{Im}E_{n_{g,h}} = -\frac{1}{2}(E_{n_{ig+h,ig+h}} - E_{n_{g,g}} -E_{n_{h,h}}).$$  

\noindent Accordingly, define $\text{Re}\mu_{g,h} = \frac{1}{2}(\mu_{g+h,g+h} - \mu_{g,g} -\mu_{h,h})$ and $\text{Im}\mu_{g,h} = -\frac{1}{2}(\mu_{ig+h,ig+h} - \mu_{g,g} -\mu_{h,h}).$  Hence, by the discussion in the above paragraph, we can conclude equation (\ref{weakconv}). 

Using a similar method as in the proof of Theorem 2.19 from our previous paper (see \cite{Davison4}), we can conclude that the map $[g,h] \mapsto \mu_{g,h}$ is sesquilinear, and that $\mu_{g,h}$ inherits the following two additional properties: 

\begin{itemize} 

\item For $g,h \in \mathcal{H}$, $\mu_{g,h}$ has total variation less than or equal to $||g||||h||$. 

\item For $g,h \in \mathcal{H}$, $\overline{\mu_{g,h}} = \mu_{h,g}$. 

\end{itemize}

Let $\Delta \in \mathscr{B}(Y)$.  The map $[g,h] \mapsto \int_Y {\bf{1}}_{\Delta} d\mu_{g,h}$ is a bounded sesquilinear form with bound $1$.  By Theorem \ref{boundedsesqui}, there exists a unique bounded operator, $E(\Delta) \in \mathcal{B}(\mathcal{H})$, such that for all $g,h \in \mathcal{H}$
$$\langle E(\Delta)g, h \rangle = \int_Y {\bf{1}}_{\Delta}  d\mu_{g,h},$$

\noindent with $||E(\Delta)|| \leq 1$.  Define $E: \mathscr{B}(Y) \rightarrow \mathcal{B}(\mathcal{H})$ by $\Delta \mapsto E(\Delta)$, and note that for $g, h \in \mathcal{H}$, $E_{g,h} = \mu_{g,h}$.  This map $E$ is a positive operator-valued measure. The proof this fact can be found in Theorem 2.19 in our previous paper \cite{Davison4}. It remains to show that $E \in P_0(Y)$.  That is, we need to show: 

\begin{enumerate} 

\item For $f \in \text{Lip}(Y)$, there exists an $M_{f,E} < \infty$ such that $\left| \int_Y f dE_{g,h} \right| \leq M_{f,E}||g||||h||$ for all $g,h \in \mathcal{H}$. 

\item $\{E_n\}_{n=1}^{\infty}$ converges to $E$ in the $\rho$ metric. 

\item $E$ is a projection-valued measure. 

\end{enumerate} 

\noindent We will first show $(1)$.  Choose some $f \geq 0 \in \text{Lip}(Y)$.  There exists a $T>0$ such that $\frac{1}{T}f \in \text{Lip}_1(Y)$.  Since we are assuming the sequence $\{E_n\}_{n=1}^{\infty}$ is Cauchy in the $\rho$ metric, the sequence of self-adjoint operators $\left\{ \int f dE_n \right\}_{n=1}^{\infty}$ is Cauchy in the operator norm topology on $\mathcal{B}(\mathcal{H})$.   This implies that these operators are uniformly bounded.  That is, there exists an $N>0$ such that
$$\left| \left| \int_Y f dE_n \right| \right| \leq N$$

\noindent for all $n = 1, ... \infty$.   

For all $g, h \in \mathcal{H}$, we claim that$\left| \int_Y f dE_{g,h} \right| < \infty.$  Indeed, choose $h \in \mathcal{H}$ and consider the sequence $\{E_{n_{h,h}}\}_{n=1}^{\infty}$ which converges to $E_{h,h}$ in the weak topology.  For $k=1,2,3,...$, define $f_k(x) = \min\{k, f(x)\} \in C_b(Y) \cap \text{Lip}(Y)$, and note that $f_k \uparrow f$ on $Y$.  We note here that the idea of using the cutoff function, $f_k$, is also a central part of the proof of Theorem \ref{Kravthm} by A. Kravchenko.  By the monotone convergence theorem
$$\int_Y f_k dE_{h,h} \uparrow \int_Y f dE_{h,h}.$$

\noindent Suppose this is an unbounded increasing sequence.  Then choose a $k_l$ such that
$$\int_Y f_{k_l} dE_{h,h} > l,$$

\noindent where $l = 1,2,...$.  For a fixed $l$,
$$\int_Y f_{k_l} dE_{n_{h,h}} \rightarrow \int_Y f_{k_l} dE_{h,h}, $$

\noindent because $\{E_{n_{h,h}}\}_{n=1}^{\infty}$ converges to $E_{h,h}$ in the weak topology and $f_{k_l} \in C_b(Y)$.  Hence choose an $n_l$ such that
$$\int_Y f_{k_l} dE_{n_{l_{h,h}}} > l.$$  Again by the monotone convergence theorem 
$$\int_Y f_{k_l} dE_{n_{l_{h,h}}} \uparrow \int_Y f dE_{n_{l_{h,h}}},$$ 

\noindent and hence, 
$$\int_Y f dE_{n_{l_{h,h}}} > l.$$  

\noindent This last line is a contradiction to the fact that the sequence $\left\{ \int_Y f dE_{n_{h,h}} \right\}_{n=1}^{\infty}$ is a Cauchy sequence of real numbers (because $f \in \text{Lip}(Y)$ and Claim \ref{Cauchy}) .  Hence $\int_Y f dE_{h,h} < \infty$ for all $h \in \mathcal{H}$.  For $g, h \in \mathcal{H}$, we can decompose $E_{g,h}$ into its positive measure parts, as we have done previously, to get that $\left| \int_Y f dE_{g,h} \right| < \infty.$

The next thing to note is that since $f_k(x) \leq f(x)$, $\int_Y f_k dE_n \leq \int_Y f dE_n$ for all $n$, as elements of $\mathcal{B}(\mathcal{H})$.  Hence for any $k$ and $n$, 
$$\left| \left| \int_Y f_k dE_n \right| \right| \leq \left| \left| \int_Y f dE_n \right| \right| \leq N.$$  

\noindent We are now prepared to show that there exists an $M_{f,E} < \infty$ such that
$$\left| \int_Y f dE_{g,h} \right| \leq M_{f,E} ||g||||h||,$$

\noindent for all $g,h \in \mathcal{H}$.  Let $\epsilon > 0$ and let $g, h \in \mathcal{H}$.  Since $f_k \uparrow f$ and $\left| \int_Y f dE_{g,h} \right| < \infty,$ there exists a $k$ such that
$$\left| \int_Y (f-f_k) dE_{g,h} \right| < \epsilon.$$  

\noindent Observe that

\begin{eqnarray*}
\left| \int_Y f dE_{g,h} \right| & \leq & \left| \int_Y (f-f_k) dE_{g,h} \right| + \left| \int_Y f_k dE_{g,h} \right| \\
					   & \leq & \epsilon + \lim_{n \rightarrow \infty} \left| \int_Y f_k d E_{n_{g,h}} \right|, 
					   \end{eqnarray*} 
					   
\noindent where the second inequality is because $f_k \in C_b(Y)$.  We know that for all $n$ and $k$ that
$$ \left| \int_Y f_k d E_{n_{g,h}} \right| \leq \left| \left \langle \left( \int f_k dE_n \right)g, h \right \rangle \right| \leq N||g||||h||.$$

\noindent Therefore 
$$\epsilon + \lim_{n \rightarrow \infty} \left| \int_Y f_k d E_{n_{g,h}} \right| \leq \epsilon + N||g||||h||.$$

\noindent Since $\epsilon$ is arbitrary, 
$$\left| \int_Y f dE_{g,h} \right| \leq N||g||||h||.$$

\noindent Note that $N$ does not depend on the choice of $g,h \in \mathcal{H}$.  It only depends on the choice of $f \geq 0 \in \text{Lip}(Y)$.  Hence, we can let $M_{f,E} = N$.  

For any arbitrary $f \in \text{Lip}(Y)$, decompose $f$ into its positive and negative parts; $f = f_+ - f_{-}$.  Note that $f_+$ and $f_{-}$ are both non-negative elements of $\text{Lip}(Y)$.  Let $M_{f,E} = M_{f_{+}, E} + M_{f_{-}, E}$.  For $g, h \in \mathcal{H}$, 
\begin{eqnarray*} 
\left| \int_Y f dE_{g,h} \right| & \leq & \left| \int_Y f_{+} dE_{g,h} \right| + \left| \int_Y f_{-} dE_{g,h} \right| \\
				& \leq & M_{f_{+}, E} ||g||||h|| + M_{f_{-}, E} ||g||||h|| \\
				& = & M_{f,E}|g||||h|| 
				\end{eqnarray*}

\noindent This completes the proof of $(1)$.  We will next show $(2)$.  We need to show that $E_n \rightarrow E$ in the $\rho$ metric.  Let $\epsilon > 0$.  Since $\{E_n\}_{n=1}^{\infty}$ is Cauchy in the $\rho$ metric, there exists an $N$ such that for $n, m \geq N$, $\rho(E_n, E_m) \leq \frac{\epsilon}{6}$.  Let $n \geq N$, let $f \in \text{Lip}_1(Y)$ with $f \geq 0$, and define $f_k(x) = \min\{k, f(x)\} \in C_b(Y) \cap \text{Lip}_1(Y)$.  As before, observe that $f_k \uparrow f$ on $Y$.  Let $h \in \mathcal{H}$ with $||h|| = 1$.  By the monotone convergence theorem, 
$$\int_Y f_k dE_{n_{h,h}} \uparrow \int_Y f dE_{n_{h,h}} < \infty$$

\noindent and
$$\int_Y f_k dE_{h,h} \uparrow \int_Y f dE_{h,h} < \infty,$$

\noindent where the finiteness of the limiting integrals is because $E_n \in P_0(Y)$, and because $E$ satisfies part $(1)$ above.   Accordingly, choose $k$ such that
$$\left| \int_Y f_k dE_{n_{h,h}} - \int_Y f dE_{n_{h,h}} \right| \leq \frac{\epsilon}{6}$$

\noindent and
$$\left| \int_Y f_k dE_{h,h} - \int_Y f dE_{h,h} \right| \leq \frac{\epsilon}{6}.$$

\noindent Since $\{E_{m_{h,h}}\}_{m=1}^{\infty}$ converges in the weak topology to  $E_{h,h}$, and $f_k \in C_b(Y)$, 
$$\lim_{m \rightarrow \infty} \int_Y f_k dE_{m_{h,h}} = \int_Y f_k dE_{h,h}.$$

\noindent Then
$$ \left| \left \langle \left( \int_Y f dE_n - \int_Y f dE \right)h, h \right \rangle \right|  =  \left| \int_Y f dE_{n_{h,h}} - \int_Y f dE_{h,h} \right| $$
$$ \leq  \left| \int_Y f dE_{n_{h,h}} - \int_Y f_k dE_{n_{h,h}} \right|  +  \left| \int_Y f_k dE_{n_{h,h}} - \int_Y f_k dE_{h,h} \right| +  \left| \int_Y f_k dE_{h,h} - \int_Y f dE_{h,h} \right| $$
$$ = \left| \int_Y f dE_{n_{h,h}} - \int_Y f_k dE_{n_{h,h}} \right| + \lim_{m \rightarrow \infty} \left| \int_Y f_k dE_{n_{h,h}} - \int_Y f_k dE_{m_{h,h}} \right|$$
$$ +  \left| \int_Y f_k dE_{h,h} - \int_Y f dE_{h,h} \right| $$
$$ \leq \frac{\epsilon}{6} +  \lim_{m \rightarrow \infty} \left| \int_Y f_k dE_{n_{h,h}} - \int_Y f_k dE_{m_{h,h}} \right| + \frac{\epsilon}{6} $$
$$ = \frac{\epsilon}{3} + \lim_{m \rightarrow \infty} \left| \left \langle \left( \int_Y f_k dE_n - \int_Y f_k dE_m \right)h, h \right \rangle \right| \leq \frac{\epsilon}{3} + \frac{\epsilon}{6}  =  \frac{\epsilon}{2},$$

\noindent because of the inequality
$$\left| \left \langle \left( \int_Y f_k dE_n - \int_Y f_k dE_m \right)h, h \right \rangle \right| \leq \left| \left| \int_Y f_k dE_n - \int_Y f_k dE_m \right| \right| ||h||^2$$ 
$$\leq \rho(E_n,E_m) ||h||^2 = \rho(E_n,E_m).$$  

\noindent Hence for $n \geq N$ and $f \in \text{Lip}_1(Y)$ such that $f \geq 0$,
$$\left| \left| \int_Y f dE_n - \int_Y f dE \right| \right| \leq \frac{\epsilon}{2}.$$  

Now for arbitrary $f \in \text{Lip}_1(Y)$, decompose $f$ into its positive and negative parts; $f = f_{+} - f_{-}$.  Note that $f_{+}$ and $f_{-}$ are both non-negative elements of $\text{Lip}_1(Y)$.  Then for $n \geq N$
\begin{eqnarray*}
\left| \left| \int_Y f dE_n - \int_Y f dE \right| \right| & \leq & \left| \left| \int_Y f_{+} dE_n - \int_Y f_{+} dE \right| \right| + \left| \left| \int_Y f_{-} dE_n - \int_Y f_{-} dE \right| \right| \\
& \leq & \frac{\epsilon}{2} + \frac{\epsilon}{2} \\
& = & \epsilon,
\end{eqnarray*}

\noindent which shows that $\rho(E_n, E) \leq \epsilon$.  This is because the choice of $N$ is independent of the choice of $f \in \text{Lip}_1(Y)$.  

Lastly, we need to show $(3)$.  That is, we need to show that $E$ is a projection-valued measure.  Since we know that $E$ is a positive operator-valued measure, $E(\Delta)$ is self-adjoint for all $\Delta \in \mathscr{B}(Y)$.  Hence, to show that $E$ is a projection-valued measure, it is enough to show that $E(\Delta_1 \cap \Delta_2) = E(\Delta_1)E(\Delta_2) $ for $\Delta_1, \Delta_2 \in \mathscr{B}(Y)$.   To this end, let $C$ and $D$ be closed subsets of $Y$.  Let $\{f_n\}_{n=1}^{\infty}$ be a sequence of functions in $\text{Lip}(Y)$ such that $f_n \downarrow {\bf{1}}_C$ and such that $||f_n||_{\infty} \leq 1$ for all $n=1,2..$. For instance, one could let $f_n(x) = \max\{1 - nd(x,C), 0 \}$. Similarly, let $\{g_n\}_{n=1}^{\infty}$ be a sequence of functions in $\text{Lip}(Y)$ such that $g_n \downarrow {\bf{1}}_D$ and such that $||g_n||_{\infty} \leq 1$ for all $n=1,2..$.

For all $h \in \mathcal{H}$, by the dominated convergence theorem, 
$$\left \langle \left( \int_Y f_n dE \right) h, h \right \rangle \rightarrow \left \langle \left( \int_Y {\bf{1}}_C dE \right) h, h \right \rangle$$

\noindent as $n \rightarrow \infty$.  That is, for all $h \in \mathcal{H}$
$$\int_Y f_n - {\bf{1}}_C dE_{h,h} \downarrow 0$$ 

\noindent as $n \rightarrow \infty$.  Also, note that for all $n =1,2, ...$, 
$$\left| \left| \int_Y f_n - {\bf{1}}_C dE \right| \right| \leq || f_n - {\bf{1}}_C ||_{\infty} \leq 1.$$

\noindent Moreover, since $E$ is already known to be a positive operator-valued measure, and since $f_n - {\bf{1}}_C \geq 0$ for all $n=1,2,..$, the sequence of operators $\{\int_Y f_n - {\bf{1}}_C dE \}_{n=1}^{\infty}$ are positive operators.  By the above discussion, we see that the sequence of operators $\{\int_Y f_n - {\bf{1}}_C dE \}_{n=1}^{\infty}$ satisfies Lemma \ref{pos}.  This means that for all $h \in \mathcal{H}$
$$\lim_{n \rightarrow \infty} \left| \left| \left( \int_Y f_n - {\bf{1}}_C dE \right) h \right| \right| = 0,$$

\noindent which is equivalent to saying that

\begin{equation} \label{SOT} \lim_{n \rightarrow \infty} \left| \left| \left( \int_Y f_n dE \right)h - E(C)h \right| \right| = 0. \end{equation}  

\noindent Similarly, \begin{equation} \label{SOT2} \lim_{n \rightarrow \infty} \left| \left| \left( \int_Y g_n dE \right)h - E(D)h \right| \right| = 0. \end{equation}  

\noindent We now have the following claim.

\begin{cla} \label{mult} For all $n =1,2,...$, 
$$\left(\int_Y f_n dE \right) \left( \int_Y g_n dE \right) = \int_Y f_ng_n dE.$$

\end{cla} 

\noindent Proof of Claim: Choose some $n=1,2,...$.  Since $f_n \in \text{Lip}(Y)$ with $||f_n||_{\infty} \leq 1$, and since $g_n \in \text{Lip}(Y)$ with $||g_n||_{\infty} \leq 1$, $f_n g_n \in \text{Lip}(Y)$.  Next, since $E_m \rightarrow E$ in the $\rho$ metric, we have that $\int_Y f_n dE_m \rightarrow \int_Y f_n dE$, $\int_Y g_n dE_m \rightarrow \int_Y g_n dE$, and $\int_Y f_n g_n dE_m \rightarrow \int_Y f_n g_n dE$ as $m \rightarrow \infty$ where convergence is in the operator norm.  Moreover, since $E_m$ is a projection-valued measure, and since $f_n$ and $g_n$ are bounded, we have that $\left( \int_Y f_n dE_m \right) \left( \int_Y g_n dE_m \right) = \left( \int_Y f_n g_n dE_m \right)$ for all $m =1,2,..$.  Combining all of this data, we get that 
$$\int_Y f_n g_n dE_m = \left( \int_Y f_n dE_m \right) \left( \int_Y g_n dE_m \right) \rightarrow \left( \int_Y f_n dE \right) \left( \int_Y g_n dE \right), $$

\noindent and
$$\int_Y f_n g_n dE_m \rightarrow \int_Y f_n g_n dE,$$

\noindent as $m \rightarrow \infty$ which shows that $\int_Y f_n g_ndE =  \left( \int_Y f_n dE \right) \left( \int_Y g_n dE \right)$.  This completes the proof of the claim. 

We will now show that $E(C)E(D) = E(C \cap D)$ as an operator on $\mathcal{H}$.  Note that $f_n \downarrow {\bf{1}}_C$, $g_n \downarrow {\bf{1}}_D$, and moreover, $f_n g_n \downarrow {\bf{1}}_{C \cap D}$.  Hence for $h \in \mathcal{H}$, we also have that 
$$ \left \langle \left( \int_Y f_n g_n dE \right) h, h \right \rangle \rightarrow \left \langle \left( \int_Y {\bf{1}}_{C\cap D} dE \right) h, h \right \rangle,$$

\noindent  as $n \rightarrow \infty$.  Since $E$ is a positive operator-valued measure, we know that $E(C)$ is self adjoint.  Therefore, 
$$\langle E(C)E(D) h, h \rangle = \langle E(D)h, E(C)h \rangle = \lim_{n \rightarrow \infty} \left \langle \left( \int_Y g_n dE \right)h, \left( \int_Y f_n dE \right)h \right \rangle = $$
$$\lim_{n \rightarrow \infty} \left \langle  \left( \int_Y f_n g_n dE \right)h, h \right \rangle = \left \langle \left( \int_Y {\bf{1}}_{C \cap D} dE \right) h, h \right \rangle = \langle E(C \cap D)h, h \rangle,$$

\noindent where the second equality is by equations (\ref{SOT}) and (\ref{SOT2}), and the third equality is because of Claim \ref{mult}. 

Now let $\Delta_1, \Delta_2 \in \mathscr{B}(Y)$.  If $h \in \mathcal{H}$, note that $E_{h,h}$ is a tight measure. Hence, there exists a sequence of compact subsets $\{C_k\}_{k=1}^{\infty} \subseteq \mathscr{B}(Y)$ such that $E_{h,h}(C_k) \uparrow E_{h,h}(\Delta_1)$, and a sequence of compact subsets $\{D_k\}_{k=1}^{\infty} \subseteq \mathscr{B}(Y)$ such that $E_{h,h}(D_k) \uparrow E_{h,h}(\Delta_2)$.  Additionally, $E_{h,h}(C_k \cap D_k) \uparrow E_{h,h}(\Delta_1 \cap \Delta_2)$.  Note that
$$\langle (E(\Delta_1) - E(C_k))h, h \rangle \rightarrow 0,$$ 

\noindent as $k \rightarrow \infty$.  Next, note that since $C_k \subseteq \Delta_1$ for all $k = 1, 2, ...$, the operator $E(\Delta_1) - E(C_k)$ is a positive operator.  Moreover, $||E(\Delta_1) - E(C_k)|| \leq 2$ for all $k = 1,2,...$.  We can appeal to Lemma \ref{pos} to conclude that $\lim_{k \rightarrow \infty} ||(E(\Delta_1) - E(C_k))h|| = 0$, or equivalently, $\lim_{k \rightarrow \infty} ||E(\Delta_1)h - E(C_k)h|| = 0$.  Similarly, $\lim_{k \rightarrow \infty} ||E(\Delta_2)h - E(D_k)h|| = 0$.  Then
$$\langle E(\Delta_1)E(\Delta_2) h, h \rangle = \langle E(\Delta_2)h, E(\Delta_1)h \rangle = \lim_{k \rightarrow \infty} \langle E(D_k)h, E(C_k)h \rangle = $$
$$\lim_{k \rightarrow \infty} \langle E(C_k)E(D_k)h, h \rangle = \lim_{k \rightarrow \infty} \langle E(C_k \cap D_k)h, h \rangle = \langle E(\Delta_1 \cap \Delta_2)h, h \rangle,$$

\noindent where the fourth equality is because $C_k$ and $D_k$ are compact (in particular, closed).  Hence, $E$ is a projection-valued measure, and this completes the proof of part (3), and the theorem.  

\end{proof} 

\begin{cor} \cite{Davison3} \label{corollarycomplete} [Davison] The metric space $(S_0(Y), \rho)$ is complete, and $(P_0(Y), \rho)$ is a closed subset of $(S_0(Y), \rho)$.
\end{cor} 

\begin{proof} The proof of the above theorem can be adapted to show $(S_0(Y), \rho)$ is complete. Moreover, the completeness of $(P_0(Y), \rho)$ implies that it is a closed subset of $(S_0(Y), \rho)$. 
\end{proof} 

\subsection{A Modified Generalized Kantorovich Metric} 

As mentioned above, one limitation of the Kantorovich metric, and the generalized Kantorovich metric, is that they are not finite when the underlying metric space is unbounded. We can restrict these metrics to a sub-collection of measures where they are finite, or we can introduce a modification on the these metrics as described below. 

Let $Q(Y)$ denote the collection of Borel probability measures on $Y$. In our previous paper, we define a modified Kantorovich metric, $MH$, on $Q(Y)$ defined as follows:  For $\mu, \nu \in Q(Y),$

\begin{equation} MH(\mu,\nu) = \sup \left\{ \left| \int_Y{f}d\mu - \int_Y{f}d\nu \right| : f  \in \text{Lip}_1(Y) \text{ and } ||f||_{\infty} \leq 1 \right\}. \end{equation} 

\noindent The condition $||f||_\infty \leq 1$ guarantees that $MH$ will be finite on $Q(Y)$. Using Proposition \ref{keylemma}, and Section 8.3 in Bogachev (see \cite{Bogachev}), we were able to deduce that the metric space $(Q(Y), MH)$ is complete. It is worth mentioning that this modified Kantorovich metric is not appropriate in applications to iterated function systems, and thus one must use the original definition as defined in equation (\ref{KantClass}). 

We will now modify the generalized Kantorovich metric.  Indeed, let $P(Y)$ denote the collection of projection-valued measures with respect to the pair $(Y, \mathcal{H})$.  

\begin{defi} \cite{Davison3} [Davison] Define $M\rho$ on $P(Y)$ by: 
$$M\rho(E,F) = \sup \left \{ \left| \left| \int_Y f dE - \int_Y f dF \right| \right| : f \in \text{Lip}_1(Y) \text{ and } ||f||_{\infty} \leq 1 \right \}$$

\noindent for $E,F \in P(Y)$. 

\end{defi}

Once again, the condition $||f||_{\infty} \leq 1$ guarantees that this metric will be finite on $P(Y)$.  

\begin{thm} \cite{Davison3} [Davison] The metric space $(P(Y), M\rho)$ is complete. 

\end{thm}

\begin{proof} The proof of this theorem follows the proof of Theorem \ref{P_0complete}, with several differences that we will point out.  Suppose that $\{E_n\}_{n=1}^{\infty}$ is a Cauchy sequence of elements in $(P(Y), M\rho)$.  We want to find an $E \in (P(Y), M\rho)$ such that $E_n \rightarrow E$ in the $M\rho$ metric.  Because $M\rho$ takes a supremum over $f \in \text{Lip}_1(Y)$ such that $||f||_{\infty} \leq 1$, we obtain a version of Claim \ref{Cauchy} only for $\text{Lip}_b(Y)$ functions.  However, this is not an impediment, because Proposition \ref{keylemma} only considers $\text{Lip}_b(Y)$ functions.  Hence, using the techniques of the proof of Theorem \ref{P_0complete}, we obtain a  positive operator-valued measure $E$ on $Y$.  The proof that $E$ is a projection-valued measure depends on the construction of a sequence $\{f_n\}_{n=1}^{\infty} \in \text{Lip}(Y)$, but one can see that actually this sequence of functions is contained in $\text{Lip}_b(Y)$.  Hence, the proof that $E$ is a projection-valued measure carries over to the $M\rho$ metric.  

Lastly, we need to show that $E_n \rightarrow E$ in the $M\rho$ metric. Let $\epsilon > 0$.  Choose an $N$ such that for $n,m \geq N$, $M\rho(E_n, E_m) \leq \epsilon$.  Let $f \in \text{Lip}_1(Y)$ with $||f||_{\infty} \leq 1$, and let $h \in \mathcal{H}$ with $||h|| = 1$.  If $n \geq N$
\begin{eqnarray*}
\left| \left \langle \left( \int_Y f dE_n - \int_Y f dE \right)h, h \right \rangle \right| & = & \left| \int_Y f dE_{n_{h,h}} - \int_Y f dE_{h,h} \right| \\
& = & \lim_{m \rightarrow \infty} \left| \int_Y f dE_{n_{h,h}} - \int_Y f dE_{m_{h,h}} \right|, \\
\end{eqnarray*}

\noindent where the last equality is because $f \in C_b(Y)$ and $E_{m_{h,h}}$ converges weakly to $E_{h,h}$.  Observe that for all $m \geq N$,
\begin{eqnarray*}
\left| \int_Y f dE_{n_{h,h}} - \int_Y f dE_{m_{h,h}} \right| & = & \left| \left \langle \left( \int_Y f dE_n - \int_Y f dE_m \right)h, h \right \rangle \right| \\
& \leq & M\rho(E_n, E_m) ||h||^2 \\
& \leq & \epsilon. 
\end{eqnarray*}

\noindent Therefore if $n \geq N$, $\left| \left| \int_Y f dE_n - \int_Y f dE \right| \right| \leq \epsilon.$  Since $N$ does not depend on the choice of $f$, $M\rho(E_n,E) \leq \epsilon$, and $(P(Y), M\rho)$ is complete.
\end{proof}

\begin{cor} \cite{Davison3} [Davison] Let $S(Y)$ be the collection of positive operator-valued measures with respect to $(Y, \mathcal{H})$. The metric space $(S(Y), M\rho)$ is complete, and \\  $(P(Y), M\rho)$ is a closed subset of $(S(Y), M\rho)$. 
\end{cor} 

\begin{proof} The proof of the above theorem can be adapted to the case of positive operator-valued measures. Moreover, the completeness of $(P(Y), M\rho)$ implies that it is a closed subset of $(S(Y), M\rho)$.
\end{proof} 

\subsection{WOT-weak Topology} 

For this sub-section, assume that $(Y,d)$ is compact. Thus, $P_0(Y) = P(Y)$ and $S_0(Y) = S(Y)$. In the introduction, we mentioned that the metric space $(P(Y), \rho)$ is not compact, and therefore, $(S(Y), \rho)$ is also not compact.  In this section, we will consider a topology on $S(Y)$ that is weaker than the topology induced by the $\rho$ metric, which we call the WOT-weak topology.  We will show that the WOT-weak topology on $S(Y)$ is compact, by directly generalizing the proof in the classical setting that $(Q(Y), H)$ is compact.  Importantly, we note that this fact has been previously shown by S. Ali (see \cite{Ali}), using another approach. As such, we also attribute Corollary \ref{WOTcompact} to Ali.  We remark that this section is a deviation from the main trajectory of this paper, which resumes next section, but we believe the results in this section provide some worthwhile insight into another topology on $S(Y)$. 

\begin{defi} Let $\mathcal{H}$ be a Hilbert space.  The weak operator topology (WOT) on $\mathcal{B}(\mathcal{H})$ is the locally convex topology defined by the semi norms $\{p_{h,k}: h, k \in \mathcal{H}\}$ where $p_{h,k} = |\langle Ah,k \rangle|$.  Accordingly, a net of operators $\{L_i\}_{i \in I} \subseteq \mathcal{B}(\mathcal{H})$ converges to an operator $L \in \mathcal{B}(\mathcal{H})$ in the weak operator topology if $\langle L_ih, k \rangle \rightarrow \langle Lh, k \rangle$ for all $h , k \in \mathcal{H}$. 
\end{defi} 

\begin{thm} \cite{Conway}[Proposition IX.5.5 in Conway] \label{wotcompact} If $M > 0$, the subset of $\{L \in \mathcal{B}(\mathcal{H}) : ||L|| \leq M\} \subseteq \mathcal{B}(\mathcal{H})$ is compact in the weak operator topology.
\end{thm}

Equip $\mathcal{B}(\mathcal{H})$ with the weak operator topology.  For each $f \in C(Y)$, define a mapping $\hat{f}: S(Y) \rightarrow \mathcal{B}(\mathcal{H})$ by $A \mapsto \int_Y f dA$.  We note here that we will use the following equivalent notations:
$$\hat{f}(A) = \int_Y f dA = A(f).$$

Let the WOT-weak topology be the weakest topology on $S(Y)$ that makes the collection of maps $\{\hat{f}: f \in C_{\mathbb{R}}(Y)\}$ continuous where we put the weak operator topology on $\mathcal{B}(\mathcal{H})$, and where $C_{\mathbb{R}}(Y)$ denotes the collection of real-valued continuous functions on $Y$.  In other words, a net of positive operator-valued measures $\{A_i\}_{i \in I} \subseteq S(Y)$ converges to a positive operator-valued measure $A \in S(Y)$, if for all $f \in C_{\mathbb{R}}(Y)$, $\hat{f}(A_i)$ converges to $\hat{f}(A)$ in the weak operator topology.  Since the weak operator topology is a weaker topology than the operator norm topology on
$\mathcal{B}(\mathcal{H})$, the WOT-weak topology is a weaker topology than the topology induced by the $\rho$ metric on $S(Y)$.  

\begin{thm}\label{sequential} \cite{Davison3} [Davison] The WOT-weak topology is sequentially compact. 
\end{thm} 

\begin{proof} Let $\{A_n\}_{n=1}^{\infty}$ be a sequence in $S(Y)$.  Since $Y$ is compact, $C(Y)$ is separable, and therefore choose a countable dense subset of functions $\{f_i\}_{i=1}^{\infty}\subseteq C(Y)$.  Consider the bounded operators $\{A_n(f_1)\}_{n=1}^{\infty}$.  Note that for all $n = 1, ...$, $||A_n(f_1)|| \leq ||f_1||_{\infty}$.  Since the subset $\{L \in \mathcal{B}(\mathcal{H}) : ||L|| \leq ||f_1||_{\infty} \} \subseteq \mathcal{B}(\mathcal{H})$ is compact in the weak operator topology (see Theorem \ref{wotcompact}), the sequence $\{A_n(f_1)\}_{n=1}^{\infty}$ admits a convergent subsequence in the weak operator topology, which we call $\{A_n^1(f_1)\}_{n=1}^{\infty}$.  Consider the sequence of bounded operators $\{A_n^1(f_2)\}_{n=1}^{\infty}$.  Since for all $n = 1, ... \infty$, $||A_n^1(f_2)|| \leq ||f_2||_{\infty}$, the subsequence $\{A_n^1(f_2)\}_{n=1}^{\infty}$ admits a further subsequence $\{A_n^2(f_2)\}_{n=1}^{\infty}$ which is convergent in the weak operator topology.  If we continue the process, we obtain for each $i =1, ... \infty$ a sequence $\{A_n^i(f_i)\}_{n=1}^{\infty}$ which is convergent in the weak operator topology, such that $\{A_n^{i+1}\}_{n=1}^{\infty}$ is a subsequence of $\{A_n^i\}_{n=1}^{\infty}$.  Now choose some $f_i \in C(Y)$ for $1 \leq i \leq \infty$, and consider the diagonal sequence, $\{A_n^n(f_i)\}_{n=1}^{\infty}$.  For $n \geq i$, $\{A_n^n(f_i)\}$ is a subsequence of $\{A_n^i(f_i)\}$, and since $\{A_n^i(f_i)\}_{n=i}^{\infty}$ is convergent in the weak operator topology, so is $\{A_n^n(f_i)\}_{n=i}^{\infty}$, which implies that $\{A_n^n(f_i)\}_{n=1}^{\infty}$ is convergent in the weak operator topology. 

Let $f \in C(Y)$ and $g, h \in \mathcal{H}$.   We will show that the sequence $\{ \langle A_n^n(f)g, h \rangle \}_{n=1}^{\infty}$ is Cauchy in $\mathbb{C}$.  If $g = 0$ or $h =0$, then the result is clear because every term in the sequence is zero.  Therefore, suppose that $g \neq 0$ and $h \neq 0$.  Choose $f_i \in C(Y)$ such that 
$$||f-f_i||_{\infty} \leq \frac{\epsilon}{3||h||||g||}.$$

\noindent By above, we know that  $\{A_n^n(f_i)\}_{n=1}^{\infty}$ is convergent in the weak operator topology.  Therefore, there exists an $N$ such that for $m,n \geq N$, $| \langle A_n^n(f_i)g, h \rangle - \langle A_m^m(f_i)g, h \rangle | \leq \frac{\epsilon}{3}.$  Thus, if $m,n \geq N$ 
\begin{eqnarray*}
| \langle A_n^n(f)g, h \rangle - \langle A_m^m(f)g, h \rangle | & \leq & | \langle A_n^n(f)g, h \rangle - \langle A_n^n(f_i)g, h \rangle | \\
& + & | \langle A_n^n(f_i)g, h \rangle - \langle A_m^m(f_i)g, h \rangle | \\
& + & | \langle A_m^m(f_i)g, h \rangle - \langle A_m^m(f)g, h \rangle | \\
& \leq & \int_Y |f-f_i| d{A_n^{n}}_{g,h} + \frac{\epsilon}{3} + \int_Y |f-f_i| d{A_m^{m}}_{g,h} \\
& \leq & \epsilon. 
\end{eqnarray*} 

\noindent Hence, for all $f \in C(Y)$ and $g, h \in \mathcal{H}$, the sequence 
$$\{ \langle A_n^n(f)g, h \rangle \}_{n=1}^{\infty} = \left\{ \int_Y f d{A_n^{n}}_{g,h}\right\}_{n=1}^{\infty}$$ 

\noindent is Cauchy in $\mathbb{C}$.  Define $\mu_{g,h}: C(Y) \rightarrow \mathbb{C}$ by $f \mapsto \lim_{n \rightarrow \infty} \int_Y f d{A_n^{n}}_{g,h}$.  Observe that $\mu_{g,h}$ is a bounded linear functional, and hence $\mu_{g,h}$ is a measure.  

Using a similar approach as in the proof of Theorem 2.19 from our previous paper (see \cite{Davison4}), we note that the map $[g,h] \mapsto \mu_{g,h}$ is sesquilinear, and accordingly, there exists a positive operator-valued measure $A \in S(Y)$ such that $\langle A(\Delta)g,h \rangle = \mu_{g,h}(\Delta) $ for all $\Delta \in \mathscr{B}(Y)$.  

It remains to show that $\{A_n^{n}\}_{n=1}^{\infty}$ converges to $A$ in the weak operator topology. Choose $f \in C_{\mathbb{R}}(Y)$, and $g, h \in \mathcal{H}$.  By construction, 
$$\langle A_n^{n}(f)g, h \rangle \rightarrow \langle A(f)g, h \rangle.$$

\noindent Hence, $\{A_n\}_{n=1}^{\infty}$ admits a convergent subsequence $\{A_n^n\}_{n=1}^{\infty}$ in the WOT-weak topology, which completes the proof. 

\end{proof}

\begin{prop} \cite{Davison3} [Davison] Let $\mathcal{H}$ be a separable Hilbert space.  The WOT-weak topology on $S(Y)$ is first countable.
\end{prop}

\begin{proof} Since $\mathcal{H}$ is a separable Hilbert space, let $O = \{h_j: j =1, ..., \infty \}$ be a countable orthonormal basis in $\mathcal{H}$.  Since $C_{\mathbb{R}}(Y)$ is separable, let $P$ be a countable dense subset of $C_{\mathbb{R}}(Y)$.  

Let $A \in S(Y)$, $f_1, ..., f_k \in P$, and $h_j, h_l \in O$.  For $n \in \mathbb{N} = \{1,2, ...\}$, consider the following subset of $S(Y)$: 
$$\{ B \in S(Y): | \langle B(f_i)h_j,h_l \rangle - \langle A(f_i)h_j, h_l \rangle | < \frac{1}{n}\text{ for all } i = 1,..., k\}.$$ 

\noindent Consider the collection of all finite intersections of subsets of $S(Y)$ of the above form where $A \in S(Y)$, $f_1, ..., f_k \in P$, $h_j, h_l \in O$, $n \in \mathbb{N}$ are all arbitrary.  This forms a basis for a topology on $S(Y)$ which is first countable, and let this topology be denoted $\xi$. 
 
 We claim that the the $\xi$ topology and the WOT-weak topology coincide.  To this end, put the weak operator topology on $\mathcal{B}(\mathcal{H})$, and  let $f \in C_{\mathbb{R}}(Y)$.  We will show that the previously defined map $\hat{f}:S(Y) \rightarrow \mathcal{B}(\mathcal{H})$ is continuous with respect to the $\xi$ topology.  Since the WOT-weak topology is the weakest topology making all of the  maps of the form $\{\hat{f}: f \in C_{\mathbb{R}}(Y)\}$ continuous, we will have shown that the WOT-weak topology is weaker than the $\xi$ topology. 
 
 Since the $\xi$ topology is first countable, it can be defined by sequences.  Therefore, suppose $\{A_n\}_{n=1}^{\infty} \subseteq S(Y)$ converges in the $\xi$ topology to $A \in S(Y)$.  We need to show that $\hat{f}(A_n) \rightarrow \hat{f}(A)$ in the weak operator topology.  Note that for all $n$ 
  $$\left| \left| \int f dA_n \right| \right| \leq ||f||_{\infty}$$

\noindent and hence
$$\sup_{n=1,..., \infty} \left| \left| \int f dA_n \right| \right| < \infty.$$

\noindent Therefore, by Proposition IX.1.3 in \cite{Conway}, it is enough to show that $\lim_{n \rightarrow \infty} \langle A_n(f)h_j, h_l  \rangle = \langle A(f)h_j, h_l \rangle$ for all $h_j, h_l \in O$.  Accordingly, let  $h_j, h_l \in O$, and let $\epsilon > 0$.  Choose $g \in P$ such that $||f - g||_{\infty} \leq \frac{\epsilon}{3||h_j||||h_l||}$, and choose $s>0$ such that $\frac{1}{s} \leq \frac{\epsilon}{3}$.  Consider
$$\mathcal{O} = \left \{ B \in S(Y): | \langle B(g)h_j,h_l \rangle - \langle A(g)h_j, h_l \rangle | < \frac{1}{s}\right \}.$$

\noindent Since $A_n \rightarrow A$ in the $\xi$ topology, there exists an $N$ such that for $n \geq N$, $A_n \in \mathcal{O}$.  For $n \geq N$
\begin{eqnarray*} 
| \langle A_n(f)h_j,h_l \rangle - \langle A(f)h_j, h_l \rangle | & \leq & | \langle A_n(f)h_j,h_l \rangle - \langle A(g)h_j, h_l \rangle | \\
& + & | \langle A_n(g)h_j,h_l \rangle - \langle A(g)h_j, h_l \rangle | \\
& + &  | \langle A(g)h_j,h_l \rangle - \langle A(f)h_j, h_l \rangle | \\ 
& \leq & ||f - g||_{\infty} ||h_j||||h_l|| + \frac{1}{s} + ||f-g||_{\infty}  ||h_j||||h_l|| \\
& \leq & \epsilon.
\end{eqnarray*}

\noindent Hence, $A_n(f) \rightarrow A(f)$ in the weak operator topology.  

Let $A \in S(Y)$ and let $\mathcal{W} = \{ B \in S(Y): | \langle B(f_i)h_j,h_l \rangle - \langle A(f_i)h_j, h_l \rangle | < \frac{1}{n}\text{ for all } i = 1,..., k\}$ be an arbitrary sub-basis element of the the $\xi$ topology.  We need to show that $\mathcal{W}$ is open in the WOT-weak topology.  Define $\mathcal{O}_i = \hat{f_i}^{-1}( \{ L \in \mathcal{B}(\mathcal{H}): | \langle Lh_j,h_l \rangle - \langle A(f_i)h_j, h_l \rangle | < \frac{1}{n}\})$.  Since the set  $\{ L \in \mathcal{B}(\mathcal{H}): | \langle Lh_j,h_l \rangle - \langle A(f_i)h_j, h_l \rangle | < \frac{1}{n}\}$ is open in the weak operator topology, $\mathcal{O}_i$ is open in the WOT-weak topology.  Notice that $\mathcal{O}_i = \{ B \in S(Y): | \langle B(f_i)h_j,h_l \rangle - \langle A(f_i)h_j, h_l \rangle | < \frac{1}{n}\}$.  Now observe that $\mathcal{W} = \bigcap_{i=1}^{k} \mathcal{O}_i$, which is an open element in the WOT-weak topology, because each $\mathcal{O}_i$ is open in the WOT-weak topology.  Hence, the two topologies coincide.  Since the $\xi$ topology is first countable, the WOT-weak topology is first countable as well. 

\end{proof} 

\begin{cor} \label{WOTcompact} \cite{Ali} \cite{Davison3} [Ali, Davison]  Let $\mathcal{H}$ be a separable Hilbert space.  The WOT-weak topology on $S(Y)$  is compact. 
\end{cor} 

\begin{proof} Since $\mathcal{H}$ is a separable Hilbert space, the above proposition shows that the WOT-weak topology on $S(Y)$ is first countable.  By Theorem \ref{sequential}, we know that $S(Y)$ is sequentially compact.  In first countable topologies, sequential compactness and compactness are equivalent.   Hence, $S(Y)$ with the WOT-weak topology is compact. 
\end{proof} 

\subsection{Application to IFS}

As mentioned earlier, we discuss an application to iterated function systems. Let $(Y,d)$ be an arbitrary complete and separable metric space equipped with an IFS $\mathcal{S} = \{ \sigma_0,...,\sigma_{N-1} \}$. Suppose that $\mathcal{K} = L^2(Y, \mu)$, where $\mu$ is the Hutchinson measure associated to the IFS, whose support is the compact attractor set $X \subseteq Y$. Let $S_0(Y)$ be the collection of positive operator-valued measures with respect to the pair $(Y, \mathcal{K})$, with the additional condition (see Section 2.1) that if $B \in S_0(Y)$, then for all $f \in \text{Lip}(Y)$, there exists an $0 \leq M_{f,B} < \infty$ such that 
$$\left| \int_Y f dB_{g, h} \right| \leq M_{f, B} ||f|| ||g||$$

\noindent for all $g, h \in \mathcal{K}$. Recall that \begin{center} $F_i: \mathcal{K} \rightarrow \mathcal{K}$ is given by $\displaystyle{\phi \mapsto \frac{1}{\sqrt{N}} (\phi \circ \sigma_i)}$, \end{center} 

\noindent and the $F_i$ operators satisfy 
$$\sum_{i=0}^{N-1} F_i^*F_i = \text{id}_{\mathcal{K}},$$ 

\noindent as described in Theorem \ref{Kornelson}. Note that the operators $F_i$ are originally defined on $L^2(X, \mu)$, but they can be extended to a map on $\mathcal{K}$. 

\begin{thm} \label{extenduniquepovm} [Davison] The map $V: S_0(Y) \rightarrow S_0(Y)$ given by 
$$B(\cdot) \mapsto \sum_{i=0}^{N-1} F_i^*B(\sigma_i^{-1}(\cdot))F_i$$

\noindent is a Lipschitz contraction in the $\rho$ metric. As such, there exists a unique positive operator-valued measure $A \in S_0(Y)$ satisfying 
$$A(\cdot) = \sum_{i=0}^{N-1} F_i^*A(\sigma_i^{-1}(\cdot))F_i.$$ 

\end{thm} 

\begin{proof} We need to show that $V$ maps into $S_0(Y)$. To see why $V$ is a Lipschitz contraction, we refer the reader to Theorem 2.15 in our previous paper \cite{Davison}. The correct proof of this theorem occurs in the erratum to this previous paper \cite{Davison2}. 

Let $f \in \text{Lip}(Y)$, and $B \in S_0(Y)$. One can verify that $V(B)$ satisfies the properties of a positive operator-valued measure, and it remains to check that there exists an $0 \leq M_{f, V(B)} < \infty$ such that 
$$\left| \int_Y f dV(B)_{g,h} \right| \leq M_{f, V(B)} ||g|| ||h||$$

\noindent for all $g, h \in \mathcal{H}$. Let $g, h \in \mathcal{H}$. Then

\begin{eqnarray*} 
\left | \left \langle \left( \int_Y f dV(B) \right) g, h \right \rangle \right |& = & \left| \sum_{i=0}^{N-1} \int_Y f dB_{F_ig,F_ih}(\sigma_i^{-1}(\cdot)) \right | \\
& \leq & \left | \sum_{i=0}^{N-1} \int_Y f \circ \sigma_i dB_{F_ig, F_ih} \right | \\
& \leq & \sum_{i=0}^{N-1} \left | \int_Y f \circ \sigma_i dB_{F_ig, F_ih} \right | \\
\end{eqnarray*} 

\noindent Since $f \in \text{Lip}(Y)$, $f \circ \sigma_i \in \text{Lip}(Y)$ for all $i = 0, ..., N-1$. Hence,

\begin{eqnarray*}
\sum_{i=0}^{N-1} \left | \int_Y f \circ \sigma_i dB_{F_ig, F_ih} \right | & \leq & \sum_{i=0}^{N-1} M_{f \circ \sigma_i, B} ||F_ig|| ||F_ih|| \\
& \leq & \sum_{i=0}^{N-1} M_{f \circ \sigma_i, B} ||g|| ||h|| \\
& \leq & \left( \sum_{i=0}^{N-1} M_{f \circ \sigma_i, B} \right) ||g|| ||h|| \\
& = & M_{f,B} ||g|| ||h||, \\
\end{eqnarray*}

\noindent where we define $M_{f,B} = \sum_{i=0}^{N-1} M_{f \circ \sigma_i, B}$. Note that the second inequality above is because 
\begin{eqnarray*} 
||F_ih|| = (||F_ih||^2)^{\frac{1}{2}} \leq \left( \sum_{i=0}^{N-1} ||F_ih||^2 \right)^{\frac{1}{2}} & = & \left( \sum_{i=0}^{N-1} \langle F_ih, F_ih \rangle \right)^{\frac{1}{2}} \\
& = & \left( \sum_{i=0}^{N-1} \langle F_i^*F_ih, h \rangle \right)^{\frac{1}{2}} \\
& = & \left( \left \langle \left( \sum_{i=0}^{N-1} F_i^*F_i\right)h, h \right \rangle \right)^{\frac{1}{2}} \\
& = & ( \langle h, h \rangle )^{\frac{1}{2}} \\
& = & || h ||\\
\end{eqnarray*}

By Corollary \ref{corollarycomplete}, the metric space $(S_0(Y), \rho)$ is complete, and thus there exists a unique positive operator-valued measure $A \in S_0(Y)$ satisfying 
$$A(\cdot) = \sum_{i=0}^{N-1} F_i^*A(\sigma_i^{-1}(\cdot))F_i.$$ 

\end{proof} 

We will now show that the support of $A$ is $X$. To this end, we have the following definition.

\begin{defi} The support of a positive operator-valued measure $A$ with respect to the pair $(Y, \mathcal{K})$ is the closed subset $Y \setminus \bigcup \{U \subseteq Y : A(U) = 0 \text{ and } U \text{ open} \} $, where $0$ is the zero operator on $\mathcal{K}$. 

\end{defi}

We briefly repeat some preliminaries that also appeared in our previous paper \cite{Davison4}. Define $\Gamma_N = \{0,...,N-1\}$, and let $\Omega = \prod_{1}^{\infty} \Gamma_N.$ It is well known that $\Omega$ is a compact metric space. The metric $m$ on $\Omega$ is given by

$$m(\alpha, \beta) = \frac{1}{2^j}$$ 

\noindent where $\alpha, \beta \in \Omega,$ and $j \in \mathbb{N}$ is the first entry at which $\alpha$ and $\beta$ differ. 

We next define the shift maps on this compact metric space. Indeed, for $0 \leq i \leq N-1$, let $\eta_i: \Omega \rightarrow \Omega$ be given by $\eta_i((\alpha_1, \alpha_2,...,)) = (i, \alpha_1, \alpha_2,,...,)$, and define $\eta: \Omega \rightarrow \Omega$ given by $\eta((\alpha_1,\alpha_2, \alpha_3,...)) = (\alpha_2, \alpha_3,....,)$. 

\begin{itemize} 

\item The maps $\eta_i$ are Lipschitz contractions on $\Omega$ in the $m$ metric, and therefore, the family of maps $\mathcal{T} = \{\eta_0,...,\eta_{N-1}\}$ constitutes an IFS on $\Omega$. 

\item The compact metric space $\Omega$ is itself the attractor set associated to the IFS $\mathcal{T}$.

\item The Hutchinson measure $P$ on $\Omega$ associated to the IFS $\mathcal{T}$ is called the Bernoulli measure, and it satisfies

$$P(\cdot) = \frac{1}{N} \sum_{i=0}^{N-1} P(\eta_i^{-1}(\cdot)).$$

\item The map $\eta$ is a left inverse for each $\eta_i$, meaning that $\eta \circ \eta_i = \text{id}_{\Omega}$ for each $0 \leq i \leq N-1.$ 

\end{itemize} 

For each $0 \leq i \leq N-1$ we can define $R_i: L^2(\Omega,P) \rightarrow L^2(\Omega, P)$ by 

$$\phi \mapsto \frac{1}{\sqrt{N}} (\phi \circ \eta_i).$$

\noindent By Theorem \ref{uniquepvm}, there exists a unique projection-valued measure $E$ with respect to the pair $(\Omega, L^2(\Omega,P))$ such that 

\begin{equation} \label{pvm_alphabet} E(\cdot) = \sum_{i=0}^{N-1} R_i^*E(\eta_i^{-1}(\cdot))R_i. \end{equation} 

For each $\alpha \in \Omega$, define $\pi(\alpha) = \cap_{n=1}^{\infty} \sigma_{\alpha_1} \circ ... \circ \sigma_{\alpha_n}(X),$ where $\alpha = (\alpha_1, \alpha_2, ..., \alpha_n,...)$. Since the maps $\sigma_i$ are all contractive, $\pi(\alpha)$ is a single point in $X$. Define the map $\pi: \Omega \rightarrow X \subseteq Y$ by $\alpha \rightarrow \pi(\alpha)$ as the coding map. 

\begin{lem} \cite{Kornelson} [Jorgensen, Kornelson, Shuman] The coding map is continuous. Moreover, for all $0 \leq i \leq N-1,$ we have the relation 

\begin{equation} \label{relation} \pi \circ \eta_i = \sigma_i \circ \pi. \end{equation} 

\end{lem}

This lemma is used to proved the following theorem. 

\begin{thm} \cite{Kornelson} \label{intertwining} [Jorgensen, Kornelson, Shuman]

\begin{enumerate} 

\item The operator $V: L^2(Y, \mu) \rightarrow L^2(\Omega, P)$ given by

$$V(f) = f \circ \pi$$

\noindent is isometric. 

\item The following intertwining relations hold: 

$$VF_i = R_iV,$$

\noindent for all $0 \leq i \leq N-1$. 

\end{enumerate} 

\end{thm} 

Consider now the projection-valued measure $E(\pi^{-1}(\cdot))$ from the Borel subsets of $Y$ into the projections on $L^2(\Omega, P)$. The following result first appeared in our previous paper, when we were considering positive operator-valued measures defined only on the compact attractor set $X$ \cite{Davison4}. With Theorem \ref{extenduniquepovm}, we can extend this result to positive operator-valued measures on $Y$, which becomes useful for calculating the support of $A$.  

\begin{thm} \cite{Davison4} \text{} [Davison] The projection-valued measure $E(\pi^{-1}(\cdot)),$ and the positive operator-valued measure $A$ (from Theorem \ref{extenduniquepovm}) are related as follows: 

$$V^{*}E(\pi^{-1}(\cdot))V = A(\cdot).$$

\end{thm} 

\begin{thm} \text{} [Davison] The support of the positive operator-valued measure $A$ (from Theorem \ref{extenduniquepovm}) is $X$. 
\end{thm} 

\begin{proof} Consider the open subset $U = X^C \subseteq Y$. Note that $A(U) = V^*E(\pi^{-1}(U))V = V^*E(\emptyset)V = 0$. Hence, the support of $A$ is contained in $X$. To show that $X$ is contained in the support of $A$, we employ a proof by contradiction. Indeed, suppose the point $x \in X$ is not in the support of $A$. Then $x \in W$, where $W$ is an open subset of $Y$ such that $A(W) = 0$. Since $x \in X$, and $\pi: \Omega \rightarrow X \subseteq Y$ is onto, $x = \pi(\alpha)$ for some $\alpha = (\alpha_1, \alpha_2, ..., \alpha_k, ...) \in \Omega$. In particular, there exists a $K$ such that $x \in \sigma_{\alpha_1} \circ ... \circ \sigma_{\alpha_K} (X) \subseteq W$. Therefore, $A(W) \geq A(\sigma_{\alpha_1} \circ ... \circ \sigma_{\alpha_K} (X)) = V^*E(\pi^{-1}(\sigma_{\alpha_1} \circ ... \circ \sigma_{\alpha_K} (X))V$, by the monotone property of positive operator-valued measures. Notice that $\pi^{-1}(\sigma_{\alpha_1} \circ ... \circ \sigma_{\alpha_K} (X)) \supseteq \{ \beta \in \Omega : \beta_n = \alpha_n \text{ for } 1 \leq n \leq K \}$. By an induction argument, one can show that $E(\{ \beta \in \Omega : \beta_n = \alpha_n \text{ for } 1 \leq n \leq K \}) = R_{\alpha_1}^* ... R_{\alpha_K}^*R_{\alpha_K} ... R_{\alpha_1} > 0$. Hence, $A(W) \geq V^*E(\pi^{-1}(\sigma_{\alpha_1} \circ ... \circ \sigma_{\alpha_K} (X))V \geq  V^*R_{\alpha_1}^* ... R_{\alpha_K}^*R_{\alpha_K} ... R_{\alpha_1}V > 0$, which is a contradiction to the fact that $A(W) = 0$. 

\end{proof}

\section{Acknowledgements} 

I would like to recognize my graduate advisor Judith Packer for her excellent guidance during my time in graduate school.

\end{document}